\documentclass{article}
\usepackage{amsmath,amssymb,amsthm}
\usepackage{authblk}
\usepackage{parskip}
\usepackage[numbers,longnamesfirst]{natbib}
\usepackage[margin=3.175cm]{geometry}
\usepackage{tikz}
\usepackage[bookmarks=false,colorlinks=true,citecolor=blue]{hyperref}

\newtheorem{theorem}{Theorem}[section]{\bfseries}{\itshape}
\newtheorem{lemma}[theorem]{Lemma}{\bfseries}{\itshape}
\newtheorem{conjecture}[theorem]{Conjecture}{\bfseries}{\itshape}
\newtheorem{corollary}[theorem]{Corollary}{\bfseries}{\itshape}
\newtheorem{observation}[theorem]{Observation}{\bfseries}{\itshape}
\newtheorem{claim}{Claim}{\bfseries}{\itshape}
\newtheorem{case}{Case}{\bfseries}{\itshape}
\newtheorem{subcase}{Case}[case]
\newtheorem{subsubcase}{Case}[subcase]

\title{\textbf{Spanning caterpillar in biconvex bipartite graphs}}
\author[1]{\textbf{Dhanyamol Antony}}
\author[2]{\textbf{Anita Das}}
\author[1]{\textbf{Shirish Gosavi}}
\author[1]{\textbf{Dalu Jacob}}
\author[1]{\textbf{Shashanka Kulamarva}}
\affil[1]{Department of Computer Science and Automation, Indian Institute of Science, Bengaluru, India\protect\\ Email: \texttt{\{dhanyamola, shirishgp, dalujacob, shashankak\}@iisc.ac.in}}
\affil[2]{Department of Mathematics, Manipal Institute of Technology Bengaluru, Manipal Academy of Higher Education, Manipal, India\protect\\ Email: \texttt{anita.das@manipal.edu}}
\date{}

\begin{document}

\maketitle

\begin{abstract}
    \noindent A bipartite graph $G=(A, B, E)$ is said to be a biconvex bipartite graph if there exist orderings $<_A$ in $A$ and $<_B$ in $B$ such that the neighbors of every vertex in $A$ are consecutive with respect to $<_B$ and the neighbors of every vertex in $B$ are consecutive with respect to $<_A$. A \emph{caterpillar} is a tree that will result in a path upon deletion of all the leaves. In this note, we prove that \emph{there exists a spanning caterpillar in any connected biconvex bipartite graph}. Besides being interesting on its own, this structural result has other consequences. For instance, this directly resolves the \emph{burning number conjecture} for biconvex bipartite graphs.\\
    
    
\end{abstract}

\section{Introduction}
We consider finite, simple, and undirected graphs throughout the paper. Let $G=(V,E)$ be a graph with the set of vertices $V$ and the set of edges $E$. For basic graph theoretic notations and definitions, we refer to \cite{West2001IGT}.

A graph is \emph{connected} if, for any pair of vertices, there exists a path between them. A graph is \emph{acyclic} if it does not contain a cycle. A \emph{tree} is a connected acyclic graph. A \emph{star} on $n$ vertices, denoted as $K_{1,n-1}$, is a tree with exactly $n-1$ leaves. A \emph{caterpillar} is a tree that will result in a path upon deletion of all the leaves. Henceforth, we refer to this path as \emph{residual path}. In other words, a \emph{caterpillar} is a tree having a residual path $P$ such that every vertex that is not in $P$ is adjacent to a vertex in $P$. A graph is said to be a \emph{permutation graph} if its vertices correspond to the elements of a permutation, and its edges denote the pairs of elements that are reversed by this permutation. A subgraph $H$ of a graph $G$ is said to be a \emph{spanning subgraph} if $V(H) = V(G)$. The following lemma on trees was given by \citet{Driscoll2017TreesSixCordial}.

\begin{lemma}[\cite{Driscoll2017TreesSixCordial}]\label{lem:SmallTreesAreCaterpillars}
    Every tree on at most six vertices is a caterpillar.
\end{lemma}

A graph $G=(V,E)$ is said to be \emph{bipartite} if the vertex set $V$ can be partitioned into two partite sets $A$ and $B$ such that every edge in $E$ has its one end-point in $A$ and the other end-point in $B$. A bipartite graph $G=(A,B,E)$ is said to be a \emph{convex bipartite} graph if there exists an ordering of vertices, say, $<_B$ of $B$ such that the neighbors of every vertex in $A$ are consecutive with respect to $<_B$. A bipartite graph $G=(A,B,E)$ is said to be a \emph{biconvex bipartite graph} if there exist orderings $<_A$ of $A$ and $<_B$ of $B$ such that the neighbors of every vertex in $A$ are consecutive with respect to $<_B$ and the neighbors of every vertex in $B$ are consecutive with respect to $<_A$.

For any two distinct vertices $a_i, a_j \in A$, by $a_i <_A a_j$, we mean $a_i$ appears before $a_j$ in the ordering $<_A$. Similarly, for any two distinct vertices $b_i, b_j \in B$, by $b_i <_B b_j$, we mean $b_i$ appears before $b_j$ in the ordering $<_B$. The set of all neighbors of a vertex $v$ in $G$ is denoted as $N_G(v)$ or simply $N(v)$. A bipartite graph $G=(A,B,E)$ is said to be a \emph{chain graph} if there exist orderings $<_A$ of $A$ and $<_B$ of $B$ such that $N(u) \subseteq N(v)$ whenever $u <_A v$ and $N(a) \supseteq N(b)$ whenever $a <_B b$.

By a result of \citet{Spinrad1987BipPerm}, the class of bipartite permutation graphs is a subclass of the class of biconvex bipartite graphs. Further, it is easy to see that the class of chain graphs is a subclass of the class of biconvex bipartite graphs. The following observation follows from the definition of a biconvex bipartite graph.

\begin{observation}\label{obs:BiconvexConsecutiveProperty}
    Let $G=(A,B,E)$ be a biconvex bipartite graph, $x,y \in A$ (respectively, $x,y \in B$), and $z \in N(x) \cap N(y)$. If $w \in A$ (respectively, $w \in B$) such that $x <_A w <_A y$ (respectively, $x <_B w <_B y$), then by consecutive property of the neighbors of $z$, we have that $z$ is adjacent to $w$, i.e., $z \in N(w)$.
\end{observation}

Let $G=(A,B,E)$ be a biconvex bipartite graph. Two edges $a_ib_s$ and $a_jb_r$ in $G$ are said to be \emph{cross edges} if either $a_i <_A a_j$ and $b_r <_B b_s$, or $a_j <_A a_i$ and $b_s <_B b_r$. A biconvex ordering of $G$ (pair $<_A$ and $<_B$) is said to be a \emph{straight ordering} (in short, an $S$-ordering) if, whenever a pair of cross edges $a_ib_s$ and $a_jb_r$ exist in $G$, then at least one of the edges $a_ib_r$ and $a_jb_s$ also exists. A biconvex ordering of $G$ that is also an $S$-ordering is said to be a \emph{biconvex $S$-ordering} of $G$. \citet{Abbas2000Biconvex} studied the structure of biconvex bipartite graphs and gave the following theorem.

\begin{theorem}[\cite{Abbas2000Biconvex}]\label{thm:BiconvexSOrdering}
    Every connected biconvex bipartite graph has a biconvex $S$-ordering.
\end{theorem}

A path between a pair of vertices $u$ and $v$ in $G$ is said to be a \emph{straight path} (in short, an $S$-path) if it does not contain any cross edges. \citet{Abbas2000Biconvex} also proved the following theorem on the existence of $S$-paths.

\begin{theorem}[\cite{Abbas2000Biconvex}]\label{thm:ShortestSPath}
    Let $G$ be a connected biconvex bipartite graph with a biconvex $S$-ordering and let $u$ and $v$ be any two vertices in $G$. Then there exists a shortest $(u,v)$-path in $G$ that is an $S$-path.
\end{theorem}

An \emph{asteroidal triple} in a graph $G$ is an independent set (a set of vertices with no edge between any pair) of three vertices in $G$ if every two of them have a path between them avoiding the neighborhood of the third one. If there is no asteroidal triple in a graph, then the graph is said to be \emph{asteroidal triple-free} (\emph{AT-free}, for short). By a result of \citet{Corneil1997ATFree}, we have that there exists a spanning caterpillar in every AT-free graph. One can see that biconvex bipartite graphs can have an asteroidal triple, i.e., biconvex bipartite graphs need not be AT-free. Here, we present a similar structural existential result on connected biconvex bipartite graphs. In particular, we prove the following theorem.

\begin{theorem}\label{thm:BiconBipCaterpillar}
    If $G$ is a connected biconvex bipartite graph, then there exists a spanning caterpillar in $G$.
\end{theorem}

Note that a similar statement of Theorem~\ref{thm:BiconBipCaterpillar} may not hold for convex bipartite graphs that are not biconvex. For instance, Figure~\ref{fig:NoSpanCaterpillarExample} illustrates an example of a connected convex bipartite graph that does not have a spanning caterpillar. One can verify that the graph in Figure~\ref{fig:NoSpanCaterpillarExample} is not a biconvex bipartite graph.

\begin{figure}[h]
    \centering
    \begin{tikzpicture}
    \begin{scope}[every node/.style={circle,draw,inner sep=0pt, minimum size=4ex}]
        \node (a1) at (0,3) {$a_1$};
        \node (a2) at (0,2) {$a_2$};
        \node (a3) at (0,1) {$a_3$};
        \node (a4) at (0,0) {$a_4$};
        \node (b1) at (3,3) {$b_1$};
        \node (b2) at (3,2) {$b_2$};
        \node (b3) at (3,1) {$b_3$};
    \end{scope}
    \draw (a1) to (b1);
    \draw (a1) to (b2);
    \draw (a1) to (b3);
    \draw (a2) to (b1);
    \draw (a3) to (b2);
    \draw (a4) to (b3);
    \end{tikzpicture}
    \medskip
    \caption{A convex bipartite graph having no spanning caterpillar}
    \label{fig:NoSpanCaterpillarExample}
\end{figure}
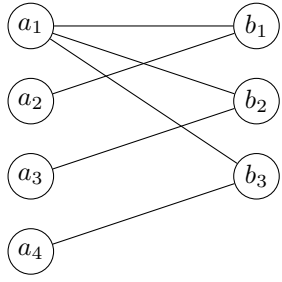

\section{Proof of Theorem~\ref{thm:BiconBipCaterpillar}}
This section is all about proving our main result, i.e., Theorem~\ref{thm:BiconBipCaterpillar}.

\begin{proof}
    Let $G$ be the given biconvex bipartite graph of order $n$, with $A$ and $B$ as the partite sets of cardinalities $n_1$ and $n_2$. Let $<_A = (a_1, a_2, \dots , a_{n_1})$ and $<_B = (b_1, b_2, \dots , b_{n_2})$ be the biconvex $S$-ordering of $G$ (existence of such a pair $<_A$ and $<_B$ follows from Theorem~\ref{thm:BiconvexSOrdering}). Suppose $n \le 6$. Since graph $G$ is connected, it has a spanning tree $T$. By Lemma~\ref{lem:SmallTreesAreCaterpillars}, we have that $T$ is a caterpillar, as desired. Thus we can assume that $n \ge 7$. Depending on the cardinalities of $n_1$ and $n_2$, we have the following cases.
    
    \begin{case}
        $n_1 = 1$ or $n_2 = 1$.
    \end{case}
    
    Without loss of generality, say $n_1 = 1$, i.e., $A=\{a_1\}$. Then since $n \ge 7$, we have $n_2 \ge 6$. Since $G$ is connected, $a_1$ is adjacent to every vertex in $B$. Since $G$ is bipartite, there is no edge between two vertices in $B$, implying that $G$ is a star, i.e., $G = K_{1,n_2}$. Hence, if we delete all the leaves in $G$ (which itself is a tree), we get $P_1$, i.e., a path on one vertex, which implies that $G$ is a caterpillar, as desired.

    \begin{case}
        $n_1 \ge 2$ and $n_2 \ge 2$
    \end{case}

    Let $a_f$ be the first neighbor of $b_1$ and let $a_l$ be the last neighbor of $b_{n_2}$ with respect to $<_A$. Note that as $n_1 \ge 2$ and $n_2 \ge 2$, we have $a_1 \neq a_{n_1}$ and $b_1 \neq b_{n_2}$. Depending on $N(a_1) \cap N(a_{n_1})$ and $N(b_1) \cap N(b_{n_2})$, we have the following cases.

    \begin{subcase}
        $N(a_1) \cap N(a_{n_1}) \neq \emptyset$ or $N(b_1) \cap N(b_{n_2}) \neq \emptyset$.
    \end{subcase}
    
    Without loss of generality, let $N(a_1) \cap N(a_{n_1}) \neq \emptyset$. Let $b_c$ be a common neighbor of $a_1$ and $a_{n_1}$. Then by Observation~\ref{obs:BiconvexConsecutiveProperty}, we have that every vertex in $A$ is adjacent to $b_c$. Now, depending on $N(b_1) \cap N(b_{n_2})$, we have the following cases.
    
    \begin{subsubcase}
        $N(b_1) \cap N(b_{n_2}) \neq \emptyset$.
    \end{subsubcase}
    
    Let $a_c$ be a common neighbor of $b_1$ and $b_{n_2}$. Then by Observation~\ref{obs:BiconvexConsecutiveProperty}, we have that every vertex in $B$ is adjacent to $a_c$. Now, we consider $P = a_cb_c$. Then every vertex of $G$ that is not in $P$ is adjacent to either $a_c$ or $b_c$ implying that $G$ has a spanning caterpillar with $P$ being the residual path.
    
    \begin{subsubcase}
        $N(b_1) \cap N(b_{n_2}) = \emptyset$.
    \end{subsubcase}
    
    Since $b_c$ is adjacent to every vertex in $A$, the fact that $N(b_1) \cap N(b_{n_2}) = \emptyset$ will imply that $b_1$, $b_c$ and $b_{n_2}$ are three distinct vertices in $B$ and $a_f \neq a_l$. Now, we consider $P = b_1a_fb_ca_lb_{n_2}$. Let $v$ be a vertex of $G$ that is not in $P$. If $v \in A$, then $v$ is adjacent to $b_c$. Otherwise, $v \in B$ and by Observation~\ref{obs:BiconvexConsecutiveProperty}, $v$ is adjacent to $a_f$ or $a_l$ implying that $G$ has a spanning caterpillar with $P$ being the residual path.

    \begin{subcase}
        $N(a_1) \cap N(a_{n_1}) = \emptyset$ and $N(b_1) \cap N(b_{n_2}) = \emptyset$.
    \end{subcase}
    
    Let $Q$ be a shortest path from $b_1$ to $b_{n_2}$ in $G$ which is also an $S$-path (recall that an $S$-path is a path without cross edges). Since $G$ is connected, the existence of $Q$ follows from Theorem~\ref{thm:ShortestSPath}. Let $Q = b'_1a'_1b'_2 \dots b'_{k-1}a'_{k-1}b'_k$ for some positive integer $k$. Here we have $b'_1 = b_1$ and $b'_k = b_{n_2}$. Moreover, every vertex of the form $a'_i$ is in $A$ and every vertex of the form $b'_i$ is in $B$. Further, since $N(b_1) \cap N(b_{n_2}) = \emptyset$, we have $a'_1 \neq a'_{k-1}$, implying that $k \ge 3$. Since $Q$ is an $S$-path, we have:
    \begin{equation*}
        a'_1 <_A a'_2 <_A \dots <_A a'_k \text{\enskip and \enskip} b_1 = b'_1 <_B b'_2 <_B \dots <_B b'_k = b_{n_2}
    \end{equation*}
    Recall that $a_f$ is the first neighbor of $b_1$ and $a_l$ is the last neighbor of $b_{n_2}$. If $a_f <_A a'_1$, then add the edge $a_fb_1$ (same as $a_fb'_1$) to the path $Q$ and label $a'_0 = a_f$. Note that otherwise, we have $a_f = a'_1$. If $a'_{k-1} <_A a_l$, then add the edge $a_lb_{n_2}$ (same as $a_lb'_k$) to the path $Q$ and label $a'_k = a_l$. Note that otherwise, we have $a_l = a'_{k-1}$. Let this modified path be $P$. In other words, the path $P$ is defined as follows:
    \begin{equation*}
        P =
        \begin{aligned}
            \begin{cases}
                a'_0b'_1a'_1b'_2 \dots b'_{k-1}a'_{k-1}b'_ka'_k, \text{ if } a'_0 = a_f <_A a'_1 \text{ and } a'_{k-1} <_A a_l = a'_k \medskip \\
                a'_0b'_1a'_1b'_2 \dots b'_{k-1}a'_{k-1}b'_k, \text{ if } a'_0 = a_f <_A a'_1 \text{ and } a'_{k-1} = a_l \medskip \\
                b'_1a'_1b'_2 \dots b'_{k-1}a'_{k-1}b'_ka'_k, \text{ if } a_f = a'_1 \text{ and } a'_{k-1} <_A a_l = a'_k \medskip \\
                b'_1a'_1b'_2 \dots b'_{k-1}a'_{k-1}b'_k, \text{ otherwise}
            \end{cases}
        \end{aligned}
    \end{equation*}
    Let $b_r$ be a vertex in $B$ but not in $P$. Then $b_r$ is sandwiched between two vertices $b'_i$ and $b'_{i+1}$ for some $i$ with $1 \le i \le k-1$ (i.e., $b'_i <_B b_r <_B b'_{i+1}$) such that $b'_i, b'_{i+1} \in V(P)$. Since the vertex $a'_i$ in $P$ is adjacent to both $b'_i$ and $b'_{i+1}$, by Observation~\ref{obs:BiconvexConsecutiveProperty}, $a'_i$ is also adjacent to $b_r$. Since $b_r$ was chosen arbitrarily, we can infer that every vertex of $B$ that is not in $P$ is adjacent to a vertex in $P$. Let $a_s$ be a vertex in $A$ but not in $P$ such that $a_f <_A a_s <_A a_l$. Then $a_s$ is sandwiched between two vertices $a'_i$ and $a'_{i+1}$, for some $i$ with $0 \le i \le k-1$ (i.e., $a'_i <_A a_s <_A a'_{i+1}$) such that $a'_i, a'_{i+1} \in V(P)$. Since the vertex $b'_{i+1}$ in $P$ is adjacent to both $a'_i$ and $a'_{i+1}$, by Observation~\ref{obs:BiconvexConsecutiveProperty}, $b'_{i+1}$ is also adjacent to $a_s$. Since $a_s$ was chosen arbitrarily, we can infer that every vertex of $A$ that is not in $P$, appearing in between $a_f$ and $a_l$ is adjacent to a vertex of $P$.

    Let $A_0$ be the set of all vertices in $A$ that appear before $a_f$ in the ordering $<_A$. Let $A_1$ be the set of all vertices in $A$ that appear after $a_l$ in the ordering $<_A$. Note that every vertex in $G$ that is not in $A_0 \cup A_1$ is either in $P$ or is adjacent to a vertex in $P$. Hence, if $A_0 = A_1 = \emptyset$, then $P$ itself is the residual path leading to the caterpillar. Therefore, we can assume that $A_0 \neq \emptyset$ or $A_1 \neq \emptyset$. Further, we have the following cases.
    
    \begin{subsubcase}\label{case:OneAmongA0A1IsEmpty}
        Exactly one among $A_0$ and $A_1$ is non-empty.
    \end{subsubcase}
    
    Without loss of generality, assume that $A_1 \neq \emptyset$. Thus we have $A_0 = \emptyset$. Then clearly, $a_{n_1} \in A_1$ and is the last vertex in $A_1$. Suppose $a_{n_1}$ is adjacent to some vertex $b'_i$ in $P$. Since $b'_i$ is adjacent to both $a'_{i-1}$ and $a_{n_1}$, by Observation~\ref{obs:BiconvexConsecutiveProperty}, $b'_i$ is also adjacent to every vertex in between $a'_{i-1}$ and $a_{n_1}$. In particular, $b'_i$ is adjacent to every vertex in $A_1$. Therefore, since $A_0 = \emptyset$, we are sure that $P$ is the residual path which yields the required caterpillar. Hence, we can assume that $a_{n_1}$ is not adjacent to any vertex in $P$.


    
    Let $\tilde{b}$ be the last neighbor of $a_{n_1}$ (for the case $A_0 \neq \emptyset$ and $A_1 = \emptyset$, we consider $\tilde{b}$ to be the first neighbor of $a_1$). Since $a_{n_1}$ is not adjacent to any vertex of $P$, we have $\tilde{b} \notin V(P)$. In particular, $\tilde{b} \neq b_1$ and $\tilde{b} \neq b_{n_2}$. Therefore, $\tilde{b}$ is sandwiched between two vertices $b'_i$ and $b'_{i+1}$ for some $i$ with $1 \le i \le k-1$ (i.e., $b'_i <_B \tilde{b} <_B b'_{i+1}$) such that $b'_i, b'_{i+1} \in V(P)$. Since the vertex $a'_i$ in $P$ is adjacent to both $b'_i$ and $b'_{i+1}$, by Observation~\ref{obs:BiconvexConsecutiveProperty}, $a'_i$ is also adjacent to $\tilde{b}$.
    
    Now, obtain a path $P_1$ from $P$ by replacing the vertex $b'_{i+1}$ by $\tilde{b}$. We call this a \emph{vertex replacement operation} and denote it as $R(b'_{i+1}, \tilde{b})$. Observe that any vertex in $G - V(P)$ which has $b'_{i+1}$ as the only neighbor in $P$ is either in between $a'_i$ and $a'_{i+1}$, or in $A_1$. Since the vertex $\tilde{b}$ in $P_1$ is adjacent to both $a'_i$ and $a'_{n_1}$, by Observation~\ref{obs:BiconvexConsecutiveProperty}, $\tilde{b}$ is also adjacent to every vertex in between $a'_i$ and $a'_{n_1}$. This also includes all the vertices in $A_1$ and all the vertices in between $a'_i$ and $a'_{i+1}$. Further, every vertex in $B$ is either in $P_1$ or is adjacent to some vertex in $P_1$, since $V(P_1) \cap A = V(P) \cap A$. Therefore, we have that every vertex of $G$ which is not in $P_1$ is adjacent to a vertex of $P_1$, implying that $G$ has a spanning caterpillar with $P_1$ being the residual path.
    
    \begin{subsubcase}
        Both $A_0$ and $A_1$ are non-empty.
    \end{subsubcase}

    Recall that $A_0$ is the set of all vertices in $A$ that appear before $a_f$ in the ordering $<_A$ and $A_1$ is the set of all vertices in $A$ that appear after $a_l$ in the ordering $<_A$. Let $G_0 = G - A_1$ and let $G_1 = G - A_0$. Notice that in graph $G_0$, we have $A_1 = \emptyset$ and in graph $G_1$, we have $A_0 = \emptyset$. Observe that in order to obtain $G_0$ and $G_1$ from $G$, we are deleting some consecutive vertices with respect to $<_A$. Hence, one can see that the graphs $G_0$ and $G_1$ remain biconvex bipartite. From the choice of $A_0$ and $A_1$, it is clear that the graphs $G_0$ and $G_1$ are connected. Therefore, by using the arguments in Case~\ref{case:OneAmongA0A1IsEmpty} individually, we obtain paths $P_0$ in $G_0$ and $P_1$ in $G_1$ of the corresponding caterpillars. Let $R(x_0, y_0)$ be the vertex replacement operation (refer Case~\ref{case:OneAmongA0A1IsEmpty} for the definition) performed to obtain $P_0$ from $P$ and let $R(x_1, y_1)$ be the vertex replacement operation performed to obtain $P_1$ from $P$. By the steps in Case~\ref{case:OneAmongA0A1IsEmpty}, we clearly have that $y_0$ is the first neighbor of $a_1$, $y_1$ is the last neighbor of $a_{n_1}$, $x_0 <_B y_0$, and $y_1 <_B x_1$.
    
    Now, we obtain a new path $\Tilde{P}$ by performing both the replacements $R(x_0, y_0)$ and let $R(x_1, y_1)$ in the path $P$. Observe that if $y_0 <_B y_1$, then the replacements $R_0$ and $R_1$ are independent of each other. In that case, one can see that every vertex in $G - V(\Tilde{P})$ is adjacent to a vertex of $\Tilde{P}$. This will in turn imply that $G$ has a spanning caterpillar with $\Tilde{P}$ being the residual path. So, the rest of the proof boils down to proving that $y_0 <_B y_1$.

    \begin{claim}
        The vertex $y_0$ appears before the vertex $y_1$ in the ordering $<_B$.
    \end{claim}

    \begin{proof}
        Since $N(a_1) \cap N(a_{n_1}) = \emptyset$, we have $y_0 \neq y_1$. By way of contradiction, assume that $y_1 <_B y_0$. Then since $a_1 <_A a_{n_1}$ and $y_1 <_B y_0$, we have a pair of cross edges $a_1y_0$ and $a_{n_1}y_1$. Since the ordering considered for $G$ is an $S$-ordering, we have $a_1y_1 \in E(G)$ or $a_{n_1}y_0 \in E(G)$. In the former case, $y_1$ is a common neighbor of $a_1$ and $a_{n_1}$, and in the latter case, $y_0$ is a common neighbor of $a_1$ and $a_{n_1}$. In either case, we have a contradiction to the fact that $N(a_1) \cap N(a_{n_1}) = \emptyset$. Hence, the claim.
    \end{proof}
    \renewcommand\qedsymbol{$\blacksquare$}

    This also completes the proof of Theorem~\ref{thm:BiconBipCaterpillar}.
\end{proof}

\section{Impact on Graph Burning}
Apart from the structural significance, Theorem~\ref{thm:BiconBipCaterpillar} would be useful in several aspects. Here, we provide an implication of the result on a well-known conjecture, called the \emph{burning number conjecture}.

For a graph $G$, the \emph{burning number}, $b(G)$, is the minimum number of iterations required to inflame (or burn) the whole graph while in each iteration the fire spreads from all burned vertices to their neighbors and one additional vertex can be burned. A sequence of vertices $B = (b_1, b_2, \dots , b_k)$ is said to be a \emph{burning sequence} of $G$ if the whole graph can be burned in $k$ steps by burning the vertices in $B$ sequentially.

The concept of burning number was coined by \citet{Bonato2016Burning}. They conjectured as follows:

\begin{conjecture}[\cite{Bonato2016Burning}]\label{conj:BurningNumber}
    For any connected graph $G$ of order $n$, $b(G) \le \lceil \sqrt{n} \rceil$.
\end{conjecture}

This is called the \emph{burning number conjecture} in the literature and is widely worked upon from thereon. \citet{Hiller2021BurningCaterpillar} proved the conjecture for caterpillars.

\begin{theorem}[\cite{Hiller2021BurningCaterpillar}]\label{thm:BurningBoundCaterpillar}
    If $G$ is a caterpillar of order $n$, then $b(G) \le \lceil \sqrt{n} \rceil$.
\end{theorem}

It is clear that burning a spanning tree of a graph is sufficient to burn the entire graph. Therefore, we have the following corollary of Theorem~\ref{thm:BiconBipCaterpillar} and Theorem~\ref{thm:BurningBoundCaterpillar}.

\begin{corollary}\label{cor:BiconBipBound}
    If $G$ is a connected biconvex bipartite graph of order $n$, then $b(G) \le \lceil \sqrt{n} \rceil$.
\end{corollary}

Hence, we can conclude that Conjecture~\ref{conj:BurningNumber} is true for the class of biconvex bipartite graphs. Consequently, the conjecture is true for any subclass of biconvex bipartite graphs. In particular, we can also conclude that Conjecture~\ref{conj:BurningNumber} is true for the bipartite permutation graphs and the chain graphs.

\section{Conclusion}
In this paper, we proved the existence of a spanning caterpillar in connected biconvex bipartite graphs. This structural existence is useful in studying several problems on biconvex bipartite graphs. One such instance is provided in the paper by applying this result in proving the burning number conjecture for biconvex bipartite graphs. It would be interesting to see whether some other problems can be addressed by using this result.

\bibliography{main}
\bibliographystyle{SK}

\end{document}